\documentclass[oneside, reqno]{amsart}

\pdfoutput=1

\usepackage{amsmath, amssymb, amsthm, graphicx}
\usepackage{mathtools}

\graphicspath{{Images/}}

\newtheorem{definition}{Definition}
\newtheorem{theorem}[definition]{Theorem}

\newtheorem{remark}[definition]{Remark}

\begin{document}

\title[G\r arding Nuij Sequences]{Nuij Sequences for G\r arding Hyperbolic Polynomials}
\author[T. Hasson]{Tristan Hasson}
\address{Tristan Hasson, Department of Mathematical Sciences, University of Durham, Durham DH1 3LE, United Kingdom}
\email{tristan.hasson@durham.ac.uk}
\date{\today}
\begin{abstract}
Hyperbolic polynomials were first introduced by G\r arding in 1959 in \cite{Garding}. The space of this type of polynomial was studied by Nuij and various results were given in \cite{Nuij}. Following the work of Kurdyka and Paunescu in \cite{Kurdyka}, this paper gives a result on hyperbolicity preserving operators on the space of polynomials.
\end{abstract}
\maketitle

\section*{Main Result}

We begin by giving the definition of a hyperbolic polynomial to fix notation.
\begin{definition}
Let $P(x_1, \dots, x_n)$ be a degree $m$ homogeneous polynomial on $\mathbb{R}^n$. If for some $e\in \mathbb{R}^n$ the univariate polynomial $t\mapsto P(x+te)$ has only real roots for all $x \in \mathbb{R}^n$, then we say $P$ is hyperbolic with respect to $e$.
\end{definition}
Nuij makes important use of a collection of hyperbolicity preserving operators in his paper \cite{Nuij} to prove results on the space of hyperbolic polynomials. Assume that $P$ is hyperbolic with respect to the $x_n$ direction, then these operators are of the form:
\begin{equation}
T_{k,s}P \coloneqq P+sx_k \frac{\partial P}{\partial x_n}
\end{equation}
where $s\in\mathbb{R}$ parameterises a collection of operators. In his paper Nuij combines these first order operators to form ones with higher order derivatives. For example in a 3 dimensional setting we could have
\begin{multline}
T_{x,s} T_{x,s} T_{y,s} T_{y,s} = P + s (2x+2y) \frac{\partial P}{\partial z} + s^2 (x^2 + 4xy + y^2) \frac{\partial^2 P}{\partial z^2} \\+ s^3 (2x^2y + 2xy^2) \frac{\partial^3 P}{\partial z^3} + s^4 x^2 y^2 \frac{\partial^4 P}{\partial z^4}
\end{multline}
Note this bears resemblance to the construction of Nuij sequences from \cite{Kurdyka}. This motivates the following definition in the setting of G\r arding's hyperbolic polynomials.

\begin{definition}
Given any $P(x_1, \dots, x_n)$, a polynomial of degree $m$ hyperbolic with respect to the $x_n$ direction, a sequence $a=(a_1, \dots, a_m)$ of polynomials, where $a_k \in \mathbb{R}_k[x_1, \dots, x_{n-1}]$, is a G\r arding Nuij sequence if the polynomial
\begin{equation}
P_{a,s}(x_1, \dots, x_n) \coloneqq P(x_1, \dots, x_n) + \sum_{k=1}^m a_k s^k \frac{\partial^k P(x_1, \dots, x_n)}{\partial x_n^k}
\end{equation}
is hyperbolic with respect to $x_n$ for any $s \in \mathbb{R}$.
\end{definition}

We now prove the following theorem giving a condition for any sequence of homogenisation coefficients to give rise to a hyperbolicity preserving operator.

\begin{theorem}
A sequence $a=(a_1, \dots, a_m)$ where $a_k \in \mathbb{R}_k[x_1, \dots, x_{n-1}]$ is a G\r arding Nuij sequence if and only if
\begin{equation}
Q_a(x_1, \dots, x_n) = x_n^m + \sum_{k=1}^m a_k \frac{\partial^k(x_n^m)}{\partial x_n^k}
\end{equation}
is a hyperbolic polynomial.
\label{Main}
\end{theorem}

The proof of this theorem makes use of a result of \cite{Borcea_2009}. We must first give the definition of a stable polynomial before giving the statement of this result.

\begin{definition}
Let $p\in \mathbb{C}[z_1, \dots, z_n]$ be a degree $m$ polynomial, if for all $(z_1, \dots, z_n) \in \mathbb{C}^n$ with $Im(z_i)>0$ $1\leq i \leq n$ we have $p(z_1, \dots, z_n)\neq 0$ then we say $p$ is stable.
\end{definition}

\begin{theorem}
Let $T: \mathbb{C}_m[z] \rightarrow \mathbb{C}_m[z]$ be a linear map. We extend $T$ to a linear map $\bar{T}: \mathbb{C}_m[z,w]\rightarrow \mathbb{C}_m[z,w]$ by setting $\bar{T}(z^kw^l)\coloneqq T(z^k)w^l$. If $\bar{T}((z+w)^d)$ is a stable polynomial in two variables, then $T$ preserves stability.
\label{Stab}
\end{theorem}

We can now prove our main result. This proof is based on a proof from \cite{Kurdyka}.

\begin{proof}[Proof of Theorem \ref{Main}]
We begin with the easy direction, if $a$ is a G\r arding Nuij sequence then $Q_a$ is a hyperbolic polynomial. Now $Q_a$ is simply the Nuij sequence construction acting upon the polynomial $x_n^m$, so since $P=x_n^m$ is hyperbolic with respect to $x_n$, we have $Q_a$ is hyperbolic.

This leaves us to prove: if $Q_a$ is a hyperbolic polynomial, then $a$ is a G\r arding Nuij sequence. Define a map
\begin{equation}
T_a(P)(x_1, \dots, x_n) \coloneqq P(x_1, \dots, x_n) + \sum_{k=1}^m a_k \frac{\partial^kP(x_1, \dots, x_n)}{\partial x_n^k}.
\end{equation}
We need to show that if $Q_a$ is hyperbolic then $T_a$ preserves hyperbolicity.

Now, $T_a$ preserves hyperbolicity if for all $v=(v_1, \dots, v_n)$ fixed, the univariate polynomial given by restricting $T_a(P)$ to the line through $v$ in the direction $x_n$ has only real roots. Denote $T_a(P)$ restricted to this line by
\begin{equation}
T_{a_v}(P_v)=T(P)(v_1, \dots, v_{n-1}, x_n),
\end{equation}
so $T_{a,v}:\mathbb{C}_m[z]\rightarrow \mathbb{C}_m[z]$ is a linear map.

Explicitly we have
\begin{equation}
T_{a,v}(P-v)(x_n) = P_v(x_n) + \sum_{k=1}^m a_k(v_1, \dots, v_{n-1}) \frac{dP_v (x_n)}{dx_n}
\end{equation}
and since the $v_1, \dots, v_{n-1}$ are fixed, so is $a_k(v_1, \dots, v_{n-1})$ which we will denote $a_{k,v}$.

To show $T_a$ preserves hyperbolicity we will show $T_{a,v}((z+w)^d)=Q_a(z+w)$, to which we will be able to apply the result of Theorem \ref{Stab}.

First consider $T_{a,v}$,
\begin{equation}
\begin{split}
T_{a,v}((z+w)^d) &= T_{a,v} \Big( \sum_{i=0}^d \binom{d}{i} z^i w^{d-i} \Big) \\
&=\sum_{i=0}^d \binom{d}{i} w^{d-i} T_{a,v}(z^i).
\end{split}
\end{equation}
Now note that
\begin{equation}
T_{a,v}(z^i) = \sum_{j=0}^i a_{j,v} (z^i)^{(j)} = \sum_{j=0}^i a_{j,v} \frac{i!}{(i-j)!} z^{i-j},
\end{equation}
so plugging this into the above we obtain
\begin{multline}
\sum_{i=0}^d \binom{d}{i} w^{d-i} \Big( \sum_{j=0}^i a_{j,v} \frac{i!}{(i-j)!} z^{i-j} \Big) \\
= \sum_{i=0}^d \binom{d!}{(d-i)!i!} w^{d-i} \Big( \sum_{j=0}^i a_{j,v} \frac{i!}{(i-j)!} z^{i-j} \Big).
\end{multline}

Now we can pick out the expression that comes with each $a_{j,v}$,
\begin{equation}
\sum_{i=j}^d \frac{d!}{(d-i)!i!} \frac{i!}{(i-j)!} z^{i-j} w^{d-i}
\end{equation}
and relabelling the index gives
\begin{equation}
\sum_{k=0}^{d-j} \frac{d!}{(d-j-k)!k!} z^k w^{d-j-k}.
\end{equation}

Now we consider $Q_a(z+w)$, by definition we have
\begin{equation}
Q_a(z+w) = \sum_{i=0}^d \frac{d!}{(d-i)!} a_{i,v} (z+w)^{d-i}.
\end{equation}
The coefficient of $a_j$ here is
\begin{equation}
\begin{split}
\frac{d!}{(d-j)!} (z+w)^{d-j} &= \frac{d!}{(d-j)!} \sum_{k=0}^{d-j} \binom{d-j}{k} z^k w^{d-j-k} \\
&= \frac{d!}{(d-j)!} \sum_{k=0}^{d-j} \frac{(d-j)!}{(d-j-k)!k!} z^k w^{d-j-k} \\
&= \sum_{k=0}^{d-j} \frac{d!}{(d-j-k)!k!} z^k w^{d-j-k}.
\end{split}
\end{equation}
This is the same expression for the $a_{j,v}$ as in $T_{a,v}$ hence $T_{a,v}((z+w)^d)=Q_a(z+w)$.

Now since $Q_a(x_n)$ has only real roots, we have that $Q_a(x_n,w)$ is a stable polynomial in two variables. Hence since $T_{a,v}(x_n,w)=Q_a(x_n,w)$, we have $T_{a,v}$ preserves stability. Restricting to real polynomials, $\mathbb{R}_m[x_1, \dots, x_n]$, we have that $T_{a,v}$ preserves real-rootedness. So for all $v\in\mathbb{R}^n$ we have if $P$ restricted to the line through $v$ in the direction $x_n$ has only real roots, then $T(P)$ restricted to the same line also has only real roots. Hence $T$ preserves G\r arding hyperbolicity and $a$ is a G\r arding Nuij sequence.
\end{proof}

\begin{remark}
Note in this proof we set $s=1$ in the definition of the G\r arding Nuij sequence. A little more work is required to prove the theorem in generality.
\end{remark}

\bibliographystyle{amsplain}
\bibliography{BibFile}{}

\providecommand{\bysame}{\leavevmode\hbox to3em{\hrulefill}\thinspace}
\providecommand{\MR}{\relax\ifhmode\unskip\space\fi MR }
\providecommand{\MRhref}[2]{%
  \href{http://www.ams.org/mathscinet-getitem?mr=#1}{#2}
}
\providecommand{\href}[2]{#2}
\begin{thebibliography}{1}

\bibitem{Borcea_2009}
Julius Borcea and Petter Br\"and\'en, \emph{P\' olya-{S}chur master theorems
  for circular domains and their boundaries}, Annals of Mathematics
  \textbf{170} (2009), no.~1, 465–492.

\bibitem{Garding}
L.~G{\r a}rding, \emph{An inequality for hyperbolic polynomials}, Journal of
  Mathematics and Mechanics \textbf{8 (6)} (1959), 957--965.

\bibitem{Kurdyka}
Krzysztof Kurdyka and Laurentiu Paunescu, \emph{Nuij type pencils of hyperbolic
  polynomials}, Canadian Mathematical Bulletin \textbf{60} (2017), no.~3,
  561–570.

\bibitem{Nuij}
Wim Nuij, \emph{A note on hyperbolic polynomials}, Mathematica Scandinavica
  \textbf{23} (1968), no.~1, 69--72.

\end{thebibliography}

\end{document}